\documentclass[11pt, a4paper]{article}
\usepackage[utf8]{inputenc}
\usepackage[T1]{fontenc}
\usepackage{amsmath,amsfonts,amsthm,amssymb,dsfont,bbm}
\usepackage[hidelinks,colorlinks,linkcolor=blue,citecolor=blue]{hyperref}
\usepackage[round]{natbib}
\usepackage{authblk}

\theoremstyle{plain}
\newtheorem{theorem}{Theorem}

\newtheorem{proposition}{Proposition} 
\newtheorem{definition}{Definition}
\newtheorem{corollary}{Corollary}
\theoremstyle{definition}
\newtheorem{example}{Example}
\newtheorem{remark}{Remark}

\newcommand{\rmd}{\mathrm{d}}
\newcommand{\bbE}{\mathbb{E}}
\newcommand{\bbR}{\mathbb{R}}

\begin{document}

\title{Stone's theorem for distributional regression in Wasserstein distance}

\author[1]{Cl\'ement Dombry}
\author[2]{Thibault Modeste}
\author[1]{Romain Pic}
 
\affil[1]{Université Franche-Comté, CNRS UMR 6623, Laboratoire de Mathématiques de Besançon, F-25000 Besançon, France}
\affil[2]{Université Claude Bernard Lyon 1, CNRS UMR 5208, Institut Camille Jordan, F-69622 Villeurbanne, France}

\date{}

\maketitle

\begin{abstract}
We extend the celebrated Stone's theorem to the framework of distributional regression. More precisely, we prove that weighted empirical distribution with local probability weights satisfying the conditions of Stone's theorem provide universally consistent estimates of the conditional distributions, where the error is measured by the Wasserstein distance of order $p\geq 1$. Furthermore, for $p=1$,  we determine the minimax rates of convergence on specific classes of distributions. We finally provide some applications of these results, including  the estimation of conditional tail expectation or probability weighted moment. 
\end{abstract}

\noindent
\textbf{Keywords:} distributional regression, Wasserstein distance, non parametric regression, minimax rate of convergence.\\
\textbf{MSC 2020 subject classification: 62G05.} 

\pagebreak
\tableofcontents

\section{Introduction}
Forecast is a major task from statistics and often of crucial importance for decision making. In the simple case when the quantity of interest is univariate and quantitative, point forecast often takes the form of regression where one aims at estimating the conditional mean (or the conditional quantile) of the response variable $Y$  given the available information encoded in a vector of covariates $X$. A point forecast is only a rough summary statistic and should at least be accompanied with an assessment of uncertainty (e.g. standard deviation or confidence interval). Alternatively, probabilistic forecasting and distributional regression \citep{gneitingkatz} suggest to estimate the full conditional distribution of $Y$ given $X$, called the predictive distribution. 

In the last decades, weather forecast has been a major motivation for the development of probabilistic forecast. Ensemble forecasts are based on a given number of deterministic models whose parameters vary slightly in order to take into account observation errors and incomplete physical representation of the atmosphere. This leads to an ensemble of different forecasts that overall also assess the uncertainty of the forecast. Ensemble forecasts suffer from bias and underdispersion \citep{hamill_1997} and need to be statistically postprocessed in order to be improved. Different postprocessing methods have been proposed, such as Ensemble Model Output Statistics \citep{gneiting_raftery_2005},  Quantile Regression Forests \citep{taillardat_2019} or Neural Networks \citep{schulz_2021} among others. Distributional regression is now widely used beyond meteorology and recent methodological works include deep distribution regression by \cite{li_2021}, distributional random forest by \cite{Cevid_et_al_2021} or isotonic distributional regression by \cite{Henzi_et_al_2021}.

The purpose of the present paper is to provide an extension to the framework of distributional regression of the celebrated Stone's theorem \citep{Stone_1977} that states the consistency of local weight algorithm for the estimation of the regression function. The strength of Stone's theorem is that it is fully non-parametric and model-free, with very mild assumptions  that  covers many important cases such as kernel algorithms and nearest neighbor methods, see e.g. \cite{gyorfi} for more details. We prove that Stone's theorem has a natural and elegant extension to distributional regression with error measured by the Wasserstein distance of order $p\geq 1$. Our result covers not only the case of a one-dimensional output $Y\in\bbR$ where the Wasserstein distance has a simple explicit form, but also the case of a multivariate output $Y\in\bbR^d$. The use of the Wasserstein distance is motivated by recent works revealing that it is a useful and powerful tool in statistics, see e.g. the review by \cite{Panaretos_Zemel_2020}.
Besides this main result, we characterize, in the case $d=1$ and $p=1$, the optimal minimax rate of convergence on suitable classes of distributions. We also discuss implications of our results to estimate various statistics of possible interest such as the expected shortfall or the probability weighted moment. 

The structure of the paper is the following. In Section~\ref{sec:background}, we present the required background on Stone's theorem and Wasserstein spaces. Section~\ref{sec:main} gathers our main results, including the extension of Stone's theorem to distributional regression (Theorem~\ref{thm:main}), the characterization of optimal minimax rates of convergence (Theorem~\ref{thm:minimax}) and some applications (Proposition~\ref{prop:appli} and the subsequent examples). All the technical proofs are postponed to Section~\ref{sec:proofs}.

\section{Background}\label{sec:background}
\subsection{Stone's theorem}
In a regression framework, we observe  a sample $(X_i,Y_i)$, $1\leq i\leq n$, of independent copies of $(X,Y)\in\bbR^k\times\bbR^d$ with distribution $P$. Based on this sample and assuming $Y$ integrable, the goal is to estimate the regression function 
\[
r(x)=\bbE[Y|X=x],\quad x\in\bbR^k.
\]
Local average estimators take the form
\begin{equation}\label{eq:r_n}
\hat r_n(x)=\sum_{i=1}^n W_{ni}(x) Y_i
\end{equation}
with $W_{n1}(x),\ldots,W_{nn}(x)$ the \emph{local weights} at $x$.  The local weights are  assumed to be measurable functions of  $x$ and  $X_1,\ldots,X_n$ but not to depend on $Y_1,\ldots, Y_n$, that is 
\begin{equation}\label{eq:X-property}
W_{ni}(x)=W_{ni}(x;X_1,\ldots,X_n),\quad 1\leq i\leq n.
\end{equation}
For the convenience of notation, the dependency on $X_1,\ldots,X_n$ is implicit. In this paper, we focus only on the case of \emph{probability weights} satisfying
\begin{equation}\label{eq:proba_weights}
 W_{ni}(x)\geq 0,\  1\leq i\leq n,\quad \mbox{and}\quad \sum_{i=1}^n W_{ni}(x)=1.
\end{equation}
Stone's Theorem states the universal consistency of the  regression estimate in $\mathrm{L}^p$-norm. 
\begin{theorem}[\cite{Stone_1977}]\label{thm:stone}
	Assume the probability weights~\eqref{eq:proba_weights} satisfy the following three conditions:
	\begin{itemize}
	\item[i)] there is $C>0$ such that $\bbE\left[\sum_{i=1}^n W_{ni}(X)g(X_i)\right]\leq C\bbE[g(X)]$ for all $n\geq 1$ and  measurable $g:\bbR^k\to [0,+\infty)$ such that $\bbE[g(X)]<\infty$; 
	\item[ii)] for all $\varepsilon>0$, $\sum_{i=1}^n W_{ni}(X) \mathds{1}_{\{\|X_i-X\|>\varepsilon \}}\to 0$ in probability as $n\to+\infty$;
	\item[iii)] $\max_{1\leq i\leq n} W_{ni}(X)\to 0$ in probability as $n\to+\infty$.	
	\end{itemize}
	Then, for all  $p\geq 1$ and $(X,Y)\sim P$ such that $\bbE[\|Y\|^p]<\infty$, 
	\begin{equation}\label{eq:stone}
	\bbE\left[\|\hat r_n(X)-r(X)\|^p\right]\longrightarrow 0 \quad \mbox{as $n\to+\infty$}.
	\end{equation}
	Conversely, if Equation~\eqref{eq:stone} holds, then the probability weights must satisfy conditions $i)-iii)$.
\end{theorem}

\begin{remark}
Stone's theorem is usually stated in dimension $d=1$. Since the convergence of random vectors $\hat r_n(X)\to r(X)$ in $\mathrm{L}^p$ is equivalent to convergence in $\mathrm{L}^p$ of all the components, the extension to the dimension $d\geq 2$ is straightforward. Furthermore, more general weights than probability weights can be considered: condition~\eqref{eq:proba_weights} can be dropped and replaced by the weaker assumptions that
\[
|W_{ni}(X)|\leq M \quad \mbox{a.s. for some $M>0$.}
\]
and 
\[
\sum_{i=1}^n W_{ni}(X)\to 1 \mbox{  in probability}.
\]
Such general weights will not be considered in the present paper and we therefore stick to probability weights. The reader can refer to \cite{biau_2015} for a complete proof of Stone's theorem together with a discussion.
\end{remark}

\begin{example}\label{example1}
The following two examples of kernel weights and nearest neighbor weights are the most important ones in the literature and we refer to \cite{gyorfi} Chapter~5 and~6 respectively for more details. 
\begin{itemize}
    \item The kernel weights are defined by
    \begin{equation}\label{eq:kernel-weights}
    W_{ni}(x)=\frac{K\Big(\frac{x-X_i}{h_n}\Big)}{\sum_{j=1}^n K\Big(\frac{x-X_j}{h_n}\Big)},\quad 1\leq i\leq n
    \end{equation}
    if the denominator is nonzero, and $1/n$ otherwise. Here the bandwidth $h_n>0$ depends only on the sample size $n$ and the function ${K:\mathbb{R}^k\to[0,+\infty)}$ is called a kernel.
    In this case, the estimator \eqref{eq:r_n} corresponds to the Nadaraya-Watson estimator of the regression function \citep{nadaraya_1964,watson_1964}.
    We say that $K$ is a boxed kernel if there are constants $R_2\geq R_1>0$ and $M_2\geq M_1>0$ such that
    \[
    M_1\mathds{1}_{\{\|x\|\leq R_1\}} \leq K(x)\leq     M_2\mathds{1}_{\{\|x\|\leq R_2\}},\quad x\in\mathbb{R}^k.
    \]
    Theorem~5.1 in \cite{gyorfi} states that, for a boxed kernel, the kernel weights \eqref{eq:kernel-weights} satisfy conditions $i)-iii)$ of Theorem~\ref{thm:stone} if and only if $h_n\to 0$ and $nh_n^k\to +\infty$ as $n\to+\infty$.
    \item The nearest neighbor (NN) weights are defined by
    \begin{equation}\label{eq:nearest-neighbor-weights}
    W_{ni}(x)=\begin{cases} \frac{1}{\kappa_n} & \mbox{if $X_i$ belongs to the $\kappa_n$-NN of $x$}\\ 0 &\mbox{otherwise} \end{cases},
    \end{equation}
    where the number of neighbors $\kappa_n\in\{1,\ldots,n\}$ depends only on the sample size. Recall that the $\kappa_n$-NN of $x$ within the sample $(X_i)_{1\leq i\leq n}$ are obtained by sorting the distances $\|X_i-x\|$ in increasing order and keeping the $\kappa_n$ points with the smallest distances -- as discussed in \cite{gyorfi} Chapter~6, several rules can be used to break ties such as lexicographic or random tie breaking.  Theorem~6.1 in the same reference states that the nearest neighbor  weights \eqref{eq:nearest-neighbor-weights} satisfy conditions $i)-iii)$ of Theorem~\ref{thm:stone} if and only if $\kappa_n\to +\infty$ and $\kappa_n/n\to 0$ as $n\to+\infty$.
\end{itemize}
\end{example}

\begin{example}\label{example2} Interestingly, some variants of the celebrated Breiman's Random Forest \citep{Breiman_2001} produce probability weights satisfying the assumptions of Stone's theorem. In Breiman's Random Forest, the splits involve both the covariates and the response variable so that the associated weighs $W_{ni}(x)=W_{ni}(x;(X_l,Y_l)_{1\leq l\leq n})$ are not in the form~\eqref{eq:X-property}.
\cite{scornet_2016} considers two simplified version of infinite random forest where the associated weights $W_{ni}(x)$ do not depend on the response values and satisfy the so call $X$-property, that is they are in the  form~\eqref{eq:X-property}. For totally non adaptive forests, the trees are grown thanks to a binary splitting rule that does not use the training sample and is totally random; the author shows that the  probability weights associated to the infinite forest satisfy the assumptions of Stone's theorem under the condition that the number of leaves grows to infinity at a rate smaller than $n$ and the  leaf volume tends to zero in probability (see Theorem~4.1 and its proof). For $q$-quantile forest, the binary splitting rules involves only the covariates and the author shows that the weights  associated  to the infinite forest satisfy the assumptions of Stone's theorem provided the subsampling number $a_n$ satifies $a_n\to+\infty$ and $a_n/n\to 0$ (see Theorem~5.1 and its proof).
\end{example}

\subsection{Wasserstein spaces}
We recall the definition and some elementary facts on Wasserstein spaces on $\bbR^d$. More details and further results on optimal transport and Wasserstein spaces can be found in the monograph by \cite{Villani_2009}, Chapter 6. 

For $p\geq 1$, the Wasserstein space $\mathcal{W}_p(\bbR^d)$  is defined as the set  Borel probability measures on $\bbR^d$ having a finite moment of order $p$, i.e. such that
\begin{equation}\label{eq:def-Mp}
M_p(\mu)=\Big(\int_{\bbR^d} \|y\|^p\,\mu(\rmd y)\Big)^{1/p}<\infty. 
\end{equation}
It is endowed with the distance defined, for $Q_1,Q_2\in \mathcal{W}_p(\bbR^d)$, by
\begin{equation}\label{eq:wasserstein}
\mathcal{W}_p(Q_1,Q_2)=\inf_{\pi\in \Pi(Q_1,Q_2)}\left(\int \|y_1-y_2\|^p\,\pi(\rmd y_1\rmd y_2)\right)^{1/p},
\end{equation}
where $\Pi(Q_1,Q_2)$ denotes the set of measures on $\bbR^d\times \bbR^d$ with margins $Q_1$ and $Q_2$. A couple $(Z_1,Z_2)$ of random variables with distributions $Q_1$ and $Q_2$ respectively is called a \textit{coupling}. The Wasserstein distance is thus the minimal distance $\|Z_1-Z_2\|_{\mathrm{L}^p}=\bbE[\|Z_1-Z_2\|^p]^{1/p}$ over all possible couplings. Existence of optimal couplings is ensured since $\bbR^d$ is a complete and separable metric space so that the infimum is indeed a minimum. 

Wasserstein distances are generally difficult to compute, but the case $d=1$ is the exception. A simple optimal coupling is provided by the probability inverse transform: for $i=1,2$, let $Q_i\in \mathcal{W}_p(\bbR)$, $F_i$ denotes its cumulative distribution function  and $F_i^{-1}$ its generalized inverse (quantile function). Then, starting from an uniform random variable $U\sim \mathrm{Unif}(0,1)$, an optimal coupling is given by $(Z_1,Z_2)=(F_1^{-1}(U),F_2^{-1}(U))$. Therefore, the Wasserstein distance is explicitly given by
\begin{equation}\label{eq:wasserstein1}
\mathcal{W}_p(Q_1,Q_2)=\left(\int_0^1 |F_1^{-1}(u)-F_2^{-1}(u)|^p \rmd u\right)^{1/p}.
\end{equation}
When $p=1$,  a simple change of variable yields 
\begin{equation}\label{eq:wasserstein2}
\mathcal{W}_1(Q_1,Q_2)=\int_{-\infty}^{+\infty} |F_1(u)-F_2(u)| \rmd u.
\end{equation}

\section{Main results}\label{sec:main}
\subsection{Stone's theorem for distributional regression} 
We now present the main result of the paper which is a natural extension of Stone's theorem to the framework of distributional regression. Given a distribution $(X,Y)\sim P$ on $\bbR^k\times\bbR^d$, we denote by $F$ the marginal distribution of $Y$ and by $F_x$ its conditional distribution given $X=x$. This conditional distribution can be estimated on a sample  $(X_i,Y_i)_{1 \leq i\leq n}$ of independent copies of $(X,Y)$ by the weighted empirical distribution
\begin{equation}\label{eq:def_wed}
    \hat F_{n,x}=\sum_{i=1}^n W_{ni}(x) \delta_{Y_i}
\end{equation}
where $\delta_{y}$ denotes the Dirac mass at point $y\in\bbR^d$.  For probability weights satisfying \eqref{eq:proba_weights}, $\hat F_{n,x}$ is a probability measure and can be viewed as a random element in the complete and separable space $\mathcal{W}_p(\bbR^d)$. We recall that the weights $W_{ni}(x)=W_{ni}(x;X_1,\ldots,X_n)$ implicitly depend on $X_1,\ldots,X_n$ but  not on $Y_1,\ldots,Y_n$.

\begin{theorem}\label{thm:main}
	Assume the probability weights satisfy  conditions $i)-iii)$ from Theorem~\ref{thm:stone}.
	Then, for all  $p\geq 1$ and $(X,Y)$ such that $\bbE[\|Y\|^p]<\infty$, 
	\begin{equation}\label{eq:stone-disreg}
	\bbE\big[\mathcal{W}_p^p(\hat F_{n,X},F_X)\big]\longrightarrow 0 \quad \mbox{as $n\to+\infty$}.
	\end{equation}
	Conversely, if Equation~\eqref{eq:stone-disreg} holds, then the probability weights must satisfy conditions $i)-iii)$.
\end{theorem}

It is worth noticing that 
\[
\bbE\left[\|\hat r_n(X)-r(X)\|^p\right] \leq \bbE\big[\mathcal{W}_p^p(\hat F_{n,X},F_X)\big]
\]
so that Theorem~\ref{thm:main} implies  Theorem~\ref{thm:stone} in a straightforward way. The proof of Theorem~\ref{thm:main} is postponed to Section~\ref{sec:proofs}. It first considers the case $d=1$ where the Wasserstein distance is explicitly given by formula~\eqref{eq:wasserstein1}. Then, the results is extended to higher dimension $d\geq 2$ thanks to the notion of max-sliced Wasserstein distance \citep{BG21} which allows to reduce the convergence of measures on  $\bbR^d$ to the convergence of their uni-dimensional projections (a precise statement is given in Theorem~\ref{thm:BG21} below).

\subsection{Rates of convergence}
	We next consider rates of convergence in the minimax sense. Note that similar questions and results have been established in \cite{PDNT_22}, where the second order Cram\'er's distance was considered, i.e. 
	\[
	\|\hat F_{n,X}-F_X\|_{L_2}^2=\int_{\bbR}|\hat F_{n,X}(y)-F_X(y)|^2\,\rmd y.
	\]
	We focus here on  the Wasserstein distance $\mathcal{W}_p(\hat F_{n,X},F_X)$ 	and consider only the case $d=1$ and $p=1$ which allows the explicit expression~\eqref{eq:wasserstein2}. The other cases seem harder to analyze and are beyond the scope of the present paper. Our first result considers the error in Wasserstein distance when $X=x$ is fixed.
    \begin{proposition}\label{prop:upperbound}
    Assume $d=1$ and  $(X,Y)\sim P$ such that $\bbE[|Y|]<\infty$. Then,
        \[
	\bbE \big[\mathcal{W}_1(\hat F_{n,x}, F_x)\big]\leq \bbE \Big[  \sum_{i=1}^n W_{ni}(x)\mathcal{W}_1(F_{X_i},F_{x})\Big]+ M(x)	\bbE\Big[\sum_{i=1}^n W_{ni}^2(x)\Big]^{1/2},
	\]
	where $M(x)=\int_{\bbR}\sqrt{F_x(z)(1-F_x(z))}\mathrm{d}z$.
    \end{proposition}

The first term corresponds to an approximation error due to the fact that we use a biased sample to estimate $F_x$. The more regular the model is, the smaller the approximation error is. The second term is an estimation error due to the fact that we use an empirical mean to estimate $F_x$. This estimator error is smaller if the distribution error has a lower dispersion (as measured by $M(x)$) or if  $\sum_{i=1}^n W_{ni}^2(x)$ is small. Note that in the case of nearest neighbor weights, $1/\sum_{i=1}^n W_{ni}^2(x)$ is exactly equal to $\kappa$ so that this quantity is often referred to as the \textit{effective sample size} and the estimation error is proportional to the square root of the expected reciprocal effective sample size. 

In view of Proposition~\ref{prop:upperbound}, we introduce the following classes of functions.
\begin{definition}\label{def:class-D} Let $\mathcal{D}(H,L,M)$ be the class of distributions $(X,Y)\sim P$ on $\bbR^k\times\bbR$ satisfying:
	\begin{itemize}
	\item[a)] $X\in [0,1]^k$ a.s. and  $\bbE|Y|<\infty$,
	\item[b)] for all $x,x'\in[0,1]^k$, $\mathcal{W}_1(F_x,F_{x'})\leq L\|x-x'\|^H$,
	\item[c)] for all $x\in[0,1]^k$, $\int_{\bbR} \sqrt{F_{x}(z)(1-F_{x}(z))}\,\mathrm{d} z \leq M$.
	\end{itemize}
	\end{definition}
The definition of the class together with Proposition~\ref{prop:upperbound} entails that the expected error is uniformly bounded on the class $\mathcal{D}(H,L,M)$ by
	\begin{align}
	&\bbE\Big[ \mathcal{W}_1(\hat F_{n,X}, F_X)\Big]\nonumber\\
	&\leq L\bbE \Big[  \sum_{i=1}^n W_{ni}(X)\| X_i-X\|^H\Big] + M \bbE\Big[\sum_{i=1}^n W_{ni}^2(X)\Big]^{1/2}.\label{eq:general-bound}
    \end{align}
    As a consequence, Proposition~\ref{prop:upperbound} allows to derive explicit bounds uniformly on $\mathcal{D}(H,L,M)$  for the kernel and nearest neighbor methods from Example~\ref{example1}. For the sake of simplicity, we consider the uniform kernel only.
    	\begin{corollary}\label{cor:kernel}
    Let $\hat F_{n,X}$ be given by the kernel method with uniform kernel $K(x)=\mathds{1}_{\{\|x\|\leq 1\}}$ and weights given by Equation~\eqref{eq:kernel-weights}. If $P\in\mathcal{D}(H,L,M)$, then
    \[
	\bbE\big[\mathcal{W}_1(\hat F_{n,X}, F_X)\big]
	 \leq Lh_n^H +M\sqrt{(2+1/n)c_k}(nh_n^k)^{-1/2}+Lk^{H/2}c_k(nh_n^k)^{-1}
    \]
	 with $c_k=k^{k/2}$.
    \end{corollary}
    
    \begin{corollary}\label{cor:knn}
    Let $\hat F_{n,X}$ be given by the nearest neighbor method with weights given by Equation~\eqref{eq:nearest-neighbor-weights} and assume $P\in\mathcal{D}(H,L,M)$. Then,
    \[
	 \bbE\big[\mathcal{W}_1(\hat F_{n,X}, F_X)\big]
	 \leq \begin{cases}
		 	L8^{H/2}( \kappa_n/n)^{H/2} +M\kappa_n^{-1/2}& \mbox{if } k=1,\\
		 	L\tilde{c}_k^{H/2}( \kappa_n/n)^{H/k} +M\kappa_n^{-1/2}& \mbox{if } k\geq 2,
		 \end{cases}
    \]
    where $\tilde{c}_k$  depends only on the dimension $k$ and is defined in \citet[Theorem~2.4]{biau_2015}.
    \end{corollary}
    One can see  that consistency holds --- i.e. the expected error tends to $0$ as $n\to+\infty$ --- as soon as  $h_n\to 0$ and $nh_n^k\to+\infty$ for the kernel method and  $\kappa_n/n\to 0$ and $\kappa_n\to +\infty$ for the nearest neighbor method.
    
    \medskip
	The next theorem provides the optimal minimax rate of convergence on the class $\mathcal{D}(H,L,M)$. 	We say that two sequences of positive numbers $(a_n)$ and $(b_n)$ have the same rate of convergence, noted $a_n\asymp b_n$, if the ratios $a_n/b_n$ and $b_n/a_n$ remain bounded as $n\to+\infty$. 
	\begin{theorem}\label{thm:minimax} The optimal minimax rate of convergence on the class $\mathcal{D}(H,L,M)$ is given by
	\[
	\inf_{\hat F_n} \sup_{P\in \mathcal{D}(H,L,M)}\bbE[\mathcal{W}_1(\hat F_{n,X}, F_X)]  \asymp  n^{-H/(2H+k)}.
	\]
	\end{theorem}
Theorem~\ref{thm:minimax} is the counterpart of  \citet[Theorem~1]{PDNT_22} where the minimax rate of convergence for the second order Cram\'er's distance has been considered. The strategy of proof is similar: i) we prove a lower bound by considering a suitable class of binary distributions where the error in Wasserstein distance corresponds to an absolute error in point regression  for which the minimax lower rate of convergence is known; ii) we check that the upper bound for the kernel and/or nearest neighbor algorithm has the same rate of convergence as the lower bound, which proves that the optimal minimax rate of convergence has been identified. In particular, our proof shows that the kernel method defined in Equation~\eqref{eq:kernel-weights} reaches the minimax rate of convergence in any dimension $k\geq 1$ with the choice of bandwidth $h_n\asymp n^{-1/(2H+k)}$; the nearest neighbor method defined in Equation \eqref{eq:nearest-neighbor-weights} reaches the minimax rate of convergence in any dimension $k\geq2$ with the number of neighbors $\kappa_n\asymp n^{H/(H+k/2)}$.

\begin{remark} Our estimate of the minimax rate of convergence holds only for  $d=p=1$ and we briefly discuss what can be expected in other cases. 

When $p=1$ and $d\geq 2$, one may hope to use the strong equivalence between the max-sliced Wasserstein distance and the Wasserstein distance \citep[Theorem 2.3.ii]{BG21}. This requires to estimate the expectation of a supremum over the sphere and this line of research is left for further work. 

When $p>1$, even in dimension $d=1$, it seems difficult to obtain bounds for the Wasserstein distance of order $p$ without very strong assumptions. \cite{BL19} consider the rate of convergence of the empirical distribution $\hat F_n=\frac{1}{n}\sum_{i=1}^n \delta_{Y_i}$ for an i.i.d. sample $Y_1,\ldots,Y_n$ with distribution $F$ on $\bbR$. A first consistency result (Theorem 2.14) states that $\bbE[\mathcal{W}_p^p(\hat F_n,F)]\rightarrow 0$  as soon as $F$ has a finite moment of order $p\geq 1$. Regarding rates of convergence,  they show (Corollary 3.9) that for $p=1$ the standard rate of convergence holds, i.e. $\bbE[\mathcal{W}_1(\hat F_n,F)]=O(1/\sqrt{n})$, if and only if 
\[
J_1(F)=\int_{\bbR} \sqrt{F(z)(1-F(z))}\rmd z<\infty.
\] 
On the other hand, rate of convergences for higher order $p>1$ require the condition 
\[
J_p(F)=\int_{\bbR}\frac{[F(z)(1-F(z))]^{p/2}}{f(z)^{p-1}}\rmd z<\infty,
\]
where $f$ is the density of the absolutely continuous component of $F$. They show (Corollary 5.5) that the standard rate holds, i.e. $\bbE[\mathcal{W}_p^p(\hat F_n,F)]=O(n^{-p/2})$, if and only if $J_p(F)<\infty$. However, this condition is very strong: it does not hold for the Gaussian distribution or for distributions with disconnected support. 
\end{remark}

\subsection{Applications}
We briefly illustrate Theorem~\ref{thm:main} with some applications and examples. In statistics, we commonly face the following generic situation: we are interested in a summary statistic $S$ with real values, e.g. quantiles or tail expectation,  and we want to assess the effect of $X$ on $Y$ through  $S$, that is we want to assess $S_{Y\mid X=x}$. Assuming that $S$ is well-defined for distributions on $\bbR^d$ with a finite moment of order $p\geq 1$,  it can be seen as a map $S:\mathcal{W}_p(\bbR^d)\to\bbR$ and then  $S_{Y\mid X=x}=S(F_x)$ with $F_x$  the conditional distribution of $Y$ given $X=x$. A natural plug-in estimate of $S_{Y\mid X=x}$ is 
\[
\hat S_{n,x}=S(\hat F_{n,x}) \quad \mbox{with $\hat F_{n,x}$ defined by \eqref{eq:def_wed}}.
\]
In this generic situation, our extension of Stone's theorem directly implies the following proposition. Recall that $M_p(\mu)$ is defined in Equation~\eqref{eq:def-Mp}.
\begin{proposition}\label{prop:appli}
Assume $\bbE[\|Y\|^p]<\infty$ and $\mathbb{P}(F_X\in\mathcal{C})=1$ where  ${\mathcal{C}\subset \mathcal{W}_p(\bbR^d)}$ denotes the continuity set of the  statistic  $S:\mathcal{W}_p(\bbR^d)\to\bbR$. Then weak consistency holds, i.e.
\[
\hat S_{n,X}\longrightarrow S_{Y\mid X} \quad \mbox{in probability as $n\to+\infty$.}
\]
If furthermore the statistic $S$ admits a  bound of the form
\begin{equation}\label{eq:linear-bound}
|S(\mu)|\leq aM_p^q(\mu)+b,\quad \mbox{with $a,b\geq 0$ and $0<q\leq p$},
\end{equation}
then consistency holds in $\mathrm{L}^{p/q}$, i.e.
\[
\mathbb{E}\big[|\hat S_{n,X}- S_{Y\mid X}|^{p/q}\big]\longrightarrow 0 \quad \mbox{as $n\to+\infty$}
\]
\end{proposition}
 
\begin{example} (quantile). For a distribution $G$ on $\mathbb{R}$, we define the associated quantile function 
\[
G^{-1}(\alpha)=\inf\{z\in\mathbb{R}: G(z)\geq \alpha\},\quad 0<\alpha<1.
\]
It is well-known that the weak convergence $G_n\stackrel{d}\to G$ implies the quantile convergence $G_n^{-1}(\alpha)\to G^{-1}(\alpha)$ at each continuity point $\alpha$ of $G^{-1}$.
Equivalently, considering $\mathcal{P}(\mathbb{R})$ endowed with the weak convergence topology, the $\alpha$-quantile statistic $S_\alpha(G)=G^{-1}(\alpha)$   is continuous at $G$ as soon as $G^{-1}$ is continuous at $\alpha$. 

In view of this, we let $\mathcal{C}=\{G\in\mathcal{P}(\mathbb{R})\colon G^{-1} \mbox{ continuous on $(0,1)$}\}$ and assume that the conditional distribution satisfies $\mathbb{P}(F_X\in\mathcal{C})=1$. Then weak convergence holds for the conditional quantiles, i.e.
\[
\hat F_{n,X}^{-1}(\alpha)\to F_X^{-1}(\alpha)\quad \mbox{in probability}.
\]
Note that no integrability condition is needed here because we can apply Proposition~\ref{prop:appli} on the transformed data $(X_i,\tilde Y_i)_{1\leq i\leq n}$,  where $\tilde Y_i=\mathrm{tan}^{-1}(Y_i)$ is bounded so that convergence in Wasserstein distance is equivalent to weak convergence. If furthermore $Y$ is $p$-integrable,  then the bound
\begin{align*}
|S_\alpha(G)|^p&\leq \frac{1}{\alpha}\int_0^\alpha |G^{-1}(u)|^p \rmd u + \frac{1}{1-\alpha}\int_\alpha^1 |G^{-1}(u)|^p \rmd u \\
&\leq \Big(\frac{1}{\alpha}+\frac{1}{1-\alpha}\Big) M_p^p(G)
\end{align*}
implies the strengthened convergence
\[
\hat F_{n,X}^{-1}(\alpha)\to F_X^{-1}(\alpha)\quad \mbox{in $\mathrm{L}^p$}.
\]
\end{example}

\begin{example} (tail expectation) The tail expectation above level $\alpha\in (0,1)$ is the risk measure defined for $G\in\mathcal{W}_1(\bbR)$ by
\[
S_\alpha(G)=\frac{1}{1-\alpha}\int_\alpha^1 G^{-1}(u)\,\rmd u.
\]
The name  comes from the equivalent definition
\[
S_{\alpha}(G)=\bbE[Y\mid Y>G^{-1}(\alpha)],\quad Y\sim G, 
\]
which holds when $G^{-1}$ is continuous at $\alpha$. One can see that 
\begin{align*}
|S_{\alpha}(G_1)-S_{\alpha}(G_2)|&\leq \frac{1}{1-\alpha}\int_\alpha^1 |G^{-1}_1(u)-G^{-1}_1(u)|\,\rmd u\\
&\leq \frac{1}{1-\alpha}\int_0^1 |G^{-1}_1(u)-G^{-1}_2(u)|\,\rmd u\\
&= \frac{1}{1-\alpha}\mathcal{W}_1(G_1,G_2).
\end{align*}
so that $S_\alpha$ is Lipschitz continuous with respect to the Wasserstein distance $\mathcal{W}_1$. 
As a consequence, the conditional tail expectation $S_\alpha(F_x)$ can be estimated in a consistent way by the plug-in estimator $S_\alpha(\hat F_{n,x})$ since
\[
\bbE[|S_\alpha(\hat F_{n,X})-S_\alpha(F_X)|]\leq \frac{1}{1-\alpha}\bbE[\mathcal{W}_1(\hat F_{n,X},F_X)]\longrightarrow 0.
\]
\end{example}

\begin{example} (probability weighted moment) A similar result holds for the probability weighted moment of order $p,q>0$ defined by
\[
S_{p,q}(G)=\int_0^1 G^{-1}(u) u^p(1-u)^q\,\rmd u,\quad G\in\mathcal{W}_1(\bbR).
\]
(\cite{prob_weig_mome}). The name  comes from the equivalent definition
\[
S(G)=\bbE[YG(Y)^p(1-G(Y))^q],\quad Y\sim G, 
\]
which holds when $G^{-1}$ is continuous on $(0,1)$. One can again check that the statistic $S_{p,q}$ is Lipschitz continuous with respect to the Wasserstein distance $\mathcal{W}_1$ since
\begin{align*}
|S_{p,q}(G_1)-S_{p,q}(G_2)|&\leq \int_0^1 |G^{-1}_1(u)-G^{-1}_2(u)|u^p(1-u)^q\,\rmd u\\
&\leq \max_{0\leq u\leq 1}u^p(1-u)^q  \times \int_0^1 |G^{-1}_1(u)-G^{-1}_2(u)|\,\rmd u\\
&= \Big(\frac{p}{p+q}\Big)^p\Big(\frac{q}{p+q}\Big)^q\mathcal{W}_1(G_1,G_2).
\end{align*}
\end{example}

\begin{example} (covariance) We conclude with a simple example in dimension $d=2$ where the statistic of interest is the covariance between the two components of $Y=(Y_1,Y_2)$ given $X=x$.
Here, we consider  
\[
S(G)=\int_{\bbR^2}y_1y_2\, \rmd G-\int_{\bbR^2}y_1\, \rmd G \int_{\bbR^2}y_2\, \rmd G,\quad G\in\mathcal{W}_2(\bbR^2).
\]
Considering square integrable random vectors $Y=(Y_1,Y_2)$ and $Z=(Z_1,Z_2)$ with distribution $G$ and $H$ respectively, we compute
\begin{align*}
&|S(G)-S(H)|\\
&=\big|\mathrm{Cov}(Y_1,Y_2)-\mathrm{Cov}(Z_1,Z_2)\big|\\
&=\big|\mathrm{Cov}(Y_1,Y_2-Z_2)-\mathrm{Cov}(Z_1-Y_1,Z_2)\big|\\
&\leq \mathrm{Var}(Y_1)^{1/2}\mathrm{Var}(Y_2-Z_2)^{1/2}+\mathrm{Var}(Z_2)^{1/2}\mathrm{Var}(Z_1-Y_1)^{1/2}
\end{align*}
were the last line is a consequence of Cauchy-Schwartz inequality. We have the upper bounds
\[
\mathrm{Var}(Y_1)^{1/2}\leq M_2(G),\quad \mathrm{Var}(Z_2)^{1/2}\leq M_2(H)
\]
and, choosing an optimal coupling $(Y,Z)$ between $G$ and $H$,
\[
\mathrm{Var}(Z_1-Y_1)^{1/2}\leq \|Y-Z\|_{L^2}=\mathcal{W}_2(G,H),\quad \mathrm{Var}(Y_2-Z_2)^{1/2}\leq\mathcal{W}_2(G,H).
\]
Altogether, we obtain,
\[
|S(G)-S(H)|\leq \big(M_2(G)+M_2(H)\big) \mathcal{W}_2(G,H).
\]
This proves that $S$ is locally Lipschitz and hence continuous  with respect to the distance $\mathcal{W}_2$. Taking $H=\delta_0$, we obtain 
\[
|S(G)|\leq M_2(G)^2
\]
and the  bound \eqref{eq:linear-bound} holds with $q=2$. Thus Proposition~\ref{prop:appli} implies that the plug-in estimator
\[
S(\hat F_{n,x}) = \sum_{i=1}^n W_{ni}(x)Y_{1i}Y_{2i}-\sum_{i=1}^n W_{ni}(x)Y_{1i}\sum_{i=1}^n W_{ni}(x)Y_{2i}
\]
is consistent in  absolute mean  for the conditional covariance
\[
S(F_{x}) = \bbE(Y_1Y_2\mid X=x)-\bbE(Y_1\mid X=x)\bbE(Y_2\mid X=x),
\]
i.e. $\bbE[|S(\hat F_{n,X})-S(F_X) |]\longrightarrow 0$ as $n\to+\infty$.
\end{example}

\section{Proofs}\label{sec:proofs}
\subsection{Proof of Theorem~\ref{thm:main}}
\begin{proof}[Proof of Theorem~\ref{thm:main} - case $d=1$]
We first consider the case when  $Y$ is uniformly bounded and takes its values in $[-M,M]$ for some $M>0$. Then, it holds
\[
F_x(z)= \begin{cases} 0& \mbox{if } z<-M \\ 1& \mbox{if } z\geq M \end{cases}\quad \mbox{and} \quad \hat F_{n,x}(z)= \begin{cases} 0& \mbox{if } z<-M \\ 1& \mbox{if } z\geq M \end{cases}.
\]
and the generalized inverse functions (quantile functions) are bounded in absolute value by $M$.
As a consequence, 
\begin{align}
\bbE\left[\mathcal{W}_p^p(\hat F_{n,X},F_X)\right]&=
\bbE\left[\int_{0}^1 |\hat F_{n,X}^{-1}(u)-F^{-1}_X(u)|^p \rmd z \right]\nonumber\\
&\leq (2M)^{p-1} \bbE\left[\int_{0}^1 |\hat F_{n,X}^{-1}(u)-F^{-1}_X(u)| \rmd u \right]\nonumber\\
&= (2M)^{p-1}\int_{-M}^M \bbE\left[|\hat F_{n,X}(z)-F_X(z)|\right] \rmd z.\label{eq:proof1}
\end{align}
In this lines, we have used  Equations~\eqref{eq:wasserstein1} and \eqref{eq:wasserstein2} together with Fubini's theorem. 

Consider the regression model $(X,\mathds{1}_{\{Y\leq z\}})\in\bbR^d\times \bbR$ where  $z\in[-M,M]$ is fixed. The corresponding regression function is 
\[
x\mapsto \bbE[\mathds{1}_{\{Y\leq z\}}|X=x]=F_x(z)
\]
and the local weight  estimator associated with the sample $(X_i,\mathds{1}_{\{Y_i\leq z\}})$, $1\leq i\leq n$ is 
\[
x\mapsto \sum_{i=1}^n W_{ni}(x)\mathds{1}_{\{Y_i\leq z\}}=\hat F_{n,x}(z).
\]
An application of Stone's theorem with $p=1$ yields
\[
\bbE\left[|\hat F_{n,X}(z)-F_X(z)|\right]\longrightarrow 0,\quad \mbox{as $n\to+\infty$},
\]
whence we deduce, by the dominated convergence theorem,
\[
\int_{-M}^M \bbE\left[|\hat F_{n,X}(z)-F_X(z)|\right]\rmd z\longrightarrow 0.
\]
The upper bound~\eqref{eq:proof1} finally implies
\[
\bbE\left[\mathcal{W}_p^p(\hat F_{n,X},F_X)\right]\longrightarrow 0.
\]

\medskip
We next consider the general case when $Y$ is not necessarily bounded. For $M>0$, we define the truncation $Y^M$ of $Y$ by
\[
Y^M= \begin{cases} -M& \mbox{if } Y<-M \\ Y& \mbox{if } -M\leq Y< M \\ M& \mbox{if } Y\geq M \end{cases}.
\]
We define similarly $Y_1^M,\ldots,Y_n^M$ the truncations of $Y_1,\ldots,Y_n$ respectively. The conditional distribution associated with $Y^M$ is
\[
F_x^{M}(z)=\mathbb{P}(Y^M\leq z|X=x)=\begin{cases} 0& \mbox{if } z<-M \\ F_x(z)& \mbox{if } -M\leq Y< M \\ 1& \mbox{if } z\geq M \end{cases}.
\]
The local weight estimation built on the truncated sample is
\[
\hat F_{n,x}^{M}(z)=\sum_{i=1}^n W_{ni}(x)\mathds{1}_{\{Y_i^M\leq z\}}.
\]
By the triangle inequality,
\[
\mathcal{W}_p(\hat F_{n,x},F_x)\leq \mathcal{W}_p(\hat F_{n,x},\hat F_{n,x}^{M})+\mathcal{W}_p(\hat F_{n,x}^{M},F_x^{M})+\mathcal{W}_p(F_x^{M},F_x),
\]
whence we deduce
\begin{align*}
& \bbE[\mathcal{W}_p^p(\hat F_{n,x},F_x)]\\
\leq & \; 3^{p-1}\left(\bbE[\mathcal{W}_p^p(\hat F_{n,X},\hat F_{n,X}^{M})]+\bbE[\mathcal{W}_p(\hat F_{n,X}^{M},F_X^{M})]+\bbE[\mathcal{W}_p^p(F_X^{M},F_X)] \right).
\end{align*}
By the preceding result in the bounded case, for any fixed $M$, the second term converge to $0$ as $n\to +\infty$. We next focus on the first and third term.

For fixed $X=x$, there is a natural coupling between the distribution $\hat F_{n,x}$ and $\hat F_{n,x}^{M}$ given by $(Z_1,Z_2)$ such that
\[
(Z_1,Z_2)=(Y_i,Y_i^M)\quad \mbox{with probability $W_{ni}(x)$}.
\]
Clearly $Z_1\sim \hat F_{n,x}$ and $Z_2\sim \hat F_{n,x}^{M}$ and this coupling provides the upper bound 
\begin{equation}\label{eq:coupling-bound}
\mathcal{W}_p^p(\hat F_{n,x},\hat F_{n,x}^{M})\leq \|Z_1-Z_2\|_{\mathrm{L}^p}^p=\sum_{i=1}^n W_{ni}(x)|Y_i-Y_i^M|^p .
\end{equation}
Let us introduce the function $g_M(x)$ defined by
\[
g_M(x)=\bbE\left[|Y-Y^M|^p\mid X=x\right].
\]
Using the fact that, conditionally on $X_1,\ldots,X_n$, the random variables $Y_1,\ldots,Y_n$ are independent with distribution $F_{X_1},\ldots,F_{X_n}$, we deduce
\[
\bbE\left[\mathcal{W}_p^p(\hat F_{n,x},\hat F_{n,x}^{M})\right]\leq \bbE\left[\sum_{i=1}^n W_{ni}(x)g_M(X_i) \right].
\]
The condition $i)$ on the weights in Stone's Theorem then implies
\[
\bbE\left[\sum_{i=1}^n W_{ni}(X)g_M(X_i) \right]\leq C\bbE[g_M(X)].
\]
Because $|Y-Y^M|^p$ converges almost surely to $0$ as $M\to+\infty$ and is bounded by $2^p|Y|^p$ which is integrable, Lebesgue's convergence theorem implies
\[
\bbE[g_M(X)]=\bbE\left[|Y-Y^M|^p\right]\longrightarrow 0 \quad \mbox{as $M\to+\infty$}.
\]
We deduce that the first term satisfies
\[
\bbE\left[\mathcal{W}_p^p(\hat F_{n,X},\hat F_{n,X}^{M})\right]\leq C\bbE[g_M(X)]\longrightarrow 0, \quad \mbox{as $M\to +\infty$}
\]
where the convergence is uniform in $n$. 

We now consider the third term. Since $Y^M$ is obtained from $Y$ by truncation, the distribution functions and quantile functions of $Y$ and $Y^M$ are related by
\[
F_x^{M}(z)=\begin{cases} 0& \mbox{if } z<-M \\ F_x(z)& \mbox{if } -M\leq z< M \\ 1& \mbox{if } z\geq M \end{cases}
\]
and
\[
(F_x^M)^{-1}(u)=\begin{cases} -M & \mbox{if } F_x^{-1}(u)<-M \\ (F_x)^{-1}(u)& \mbox{if } -M\leq F_x^{-1}(u)< M \\ M& \mbox{if } F_x^{-1}(u)\geq M \end{cases}.
\]
As a consequence
\begin{align*}
\mathcal{W}_p^p(F_x^{M},F_x)&=\int_{0}^1|(F_x^{M})^{-1}(u)-F_x^{-1}(u)|^p\rmd u\\ &=\bbE\left[|Y^M-Y|^p\mid X=x\right]=g_M(x).
\end{align*}
We deduce
\[
\bbE\left[\mathcal{W}_p^p(F_X^{M},F_X)\right]= \bbE[g_M(X)]\longrightarrow 0, \quad \mbox{as $M\to +\infty$}
\]
where the convergence is uniform in $n$. 

We finally combine the three terms. The sum can be made smaller than any $\varepsilon>0$ by first choosing $M$ large enough so that the first and third terms are smaller than $\varepsilon/3$ and then choosing $n$ large enough so that the second term is smaller than $\varepsilon/3$. This proves Equation~\eqref{eq:stone-disreg} and concludes the proof.
\end{proof}

In order to extend the proof from $d=1$ to $d\geq 2$, we need the notion of \textit{sliced Wasserstein distance}, see \cite{BG21} for instance. Let $\mathbb{S}^{d-1}=\{u\in\bbR^d:\|u\|=1\}$ be the unit sphere in $\bbR^d$ and, for $u\in\bbR^d$, let $u_*:\bbR^d\to\bbR$ be the linear form defined by $u_*(x)=u\cdot x$. The projection in direction $u$ of a measure $\mu$ on $\bbR^d$ is defined as the pushforward  $\mu\circ u_*^{-1}$ which is a measure on $\bbR$. The inequality $|u\cdot x|\leq \|x\|$ implies that $\mu\circ u_*^{-1}\in \mathcal{W}_p(\bbR)$ for all $\mu\in \mathcal{W}_p(\bbR^d)$ and $u\in\mathbb{S}^{d-1}$. The sliced and max-sliced Wasserstein distances between  $\mu,\nu\in \mathcal{W}_p(\bbR^d)$ are then defined  respectively by
\[
S\mathcal{W}_p(\mu,\nu)=\left(\int_{\mathbb{S}^{d-1}} \mathcal{W}_p^p(\mu\circ u_*^{-1},\nu\circ u_*^{-1})\,\sigma(\rmd u)\right)^{1/p},
\]
where $\sigma$ denotes the uniform measure on $\mathbb{S}^{d-1}$ and
\[
\overline{SW}_p(\mu,\nu)=\max_{u\in\mathbb{S}^{d-1}} \mathcal{W}_p(\mu\circ u_*^{-1},\nu\circ u_*^{-1}).
\]
In plain words, the sliced and max-sliced Wasserstein distance are respectively the average and the maximum over all the $1$-dimensional Wasserstein distances between the projections of $\mu$ and $\nu$. The following result  is crucial in our proof.

\begin{theorem}[\cite{BG21}]\label{thm:BG21}
For all $p\geq 1$, $S\mathcal{W}_p$ and $\overline{SW}_p$ are distances on $\mathcal{W}_p(\bbR^d)$ which are equivalent to $\mathcal{W}_p$, i.e. for all sequence $\mu,\mu_1,\mu_2,\ldots \in \mathcal{W}_p(\bbR^d)$
\[
S\mathcal{W}_p(\mu_n,\mu)\to 0\quad \Longleftrightarrow\quad \overline{SW}_p(\mu_n,\mu)\to 0\quad \Longleftrightarrow\quad \mathcal{W}_p(\mu_n,\mu)\to 0.
\]
\end{theorem}

\begin{proof}[Proof of Theorem~\ref{thm:main} - case $d\geq 2$.]
For the sake of clarity, we divide the proof into three steps:
\begin{enumerate}
\item[1)] we  prove that the result holds in max-sliced Wasserstein distance, i.e. $\bbE[\overline{SW}_p^p(\hat F_{n,X},F_X)]\to 0$;
\item[2)] we deduce that $\mathcal{W}_p(\hat F_{n,X},F_X)\to 0$ in probability;
\item[3)] we show that the sequence $\mathcal{W}_p^p(\hat F_{n,X},F_X)$ is uniformly integrable.
\end{enumerate}
Points 2) and 3) together imply  $\bbE[\mathcal{W}_p^p(\hat F_{n,X},F_X)]\to 0$ as required.

\medskip
Step 1). For all $u\in\mathbb{S}^{d-1}$, the projection $\hat F_{n,X}\circ u_*^{-1}$ is the weighted empirical distribution
\[
\hat F_{n,X}\circ u_*^{-1}=\sum_{i=1}^nW_{ni}(X)\delta_{Y_i\cdot u}.
\]
An application of Theorem~\ref{thm:main} to the $1$-dimensional sample $(Y_i\cdot u)_{i\geq 1}$ yields
\begin{equation}\label{eq:control_fixed_u}
\bbE[\mathcal{W}_p^p(\hat F_{n,X}\circ u_*^{-1},F_X\circ u_*^{-1})]\longrightarrow 0.
\end{equation}
Note indeed that $\bbE[|Y|^p]<\infty$ implies $\bbE[|Y\cdot u|^p]<\infty$ and that the conditional laws of $Y\cdot u$ are the pushforward of those of $Y$, i.e. $\mathcal{L}(Y\cdot u\mid X)= F_X\circ u_*^{-1}$.

We next consider the max-sliced Wasserstein distance. Regularity in the direction $u\in\mathbb{S}^{d-1}$ will be useful and we recall that the Wasserstein distance between projections depends on the direction in a Lipschitz way. More precisely, according to \citet[Proposition 2.2]{BG21}, 
\[
|\mathcal{W}_p(\mu\circ u_*^{-1},\nu\circ u_*^{-1}) -\mathcal{W}_p(\mu\circ v_*^{-1},\nu\circ v_*^{-1})| \leq (M_p(\mu)+M_p(\nu))\|u-v\|,
\]
for all $\mu,\nu\in \mathcal{W}_p(\bbR^d)$ and $u,v\in\mathbb{S}^{d-1}$ (recall Equation~\eqref{eq:def-Mp} for the definition of $M_p(\mu)$, $M_p(\nu)$).

The sphere $\mathbb{S}^{d-1}$ being compact, for all $\varepsilon>0$, one can find $K\geq 1$ and $u_1,\ldots,u_K\in\mathbb{S}^{d-1}$ such that the balls $B(u_i,\varepsilon)$ with centers $u_i$ and radius $\varepsilon$ cover the sphere. Then, due to the Lipschitz property, the max-sliced Wasserstein distance is controlled by
\begin{align*}
&\overline{SW}_p(\hat F_{n,X},F_X)\\
&=\max_{u\in\mathbb{S}^{d-1}}\mathcal{W}_p^p(\hat F_{n,X}\circ u_*^{-1},F_X\circ u^{-1})\\
&\leq \max_{1\leq k\leq K}\mathcal{W}_p(\hat F_{n,X}\circ u_{k*}^{-1},F_X\circ u_{k*}^{-1})+\varepsilon(M_p(\hat F_{n,X})+M_p(F_{X})).
\end{align*}
Elevating to the $p$-th power and taking the expectation, we deduce
\begin{align*}
&\bbE\big[\overline{SW}_p^p(\hat F_{n,X},F_X)\big] \\
&\leq 3^{p-1} \bbE\big[\max_{1\leq k\leq K}\mathcal{W}_p^p(\hat F_{n,X}\circ u_{k*}^{-1},F_X\circ u_{k*}^{-1})\big]  +3^{p-1}\varepsilon^p (\bbE\big[M_p^p(\hat F_{n,X})\big] +\bbE\big[M_p^p(F_{X})\big]).
\end{align*}
 The first term converges to $0$ thanks to Eq.~\eqref{eq:control_fixed_u}, i.e. 
 \[
 \bbE[\max_{1\leq i\leq K}\mathcal{W}_p^p(\hat F_{n,X}\circ u_{i*}^{-1},F_X\circ u_{i*}^{-1})]\longrightarrow 0.
 \]
The second  term is controlled by a constant times $\varepsilon^p$ since
\[
\bbE[M_p^p(\hat F_{n,X})]=\bbE\big[\sum_{i=1}^n W_{ni}(X)\|Y_i\|^p\big]\leq C\bbE[\|Y\|^p]
\]
(by property $i)$ of the weights) and 
\[
\bbE[M_p^p(F_{X})]=\bbE\big[\bbE[\|Y\|^p\mid X]\big]=\bbE[\|Y\|^p]
\]
(by the tower property of conditional expectation). Letting $\varepsilon\to 0$, the second term can be made arbitrarily small. We deduce $\bbE[\overline{SW}_p^p(\hat F_{n,X},F_X)]\to 0$.

\medskip
Step 2). As a consequence of step 1), $\overline{SW}_p(\hat F_{n,X},F_X)\to 0$
in probability, or equivalently $\hat F_{n,X}\to F_X$ in probability in the
 metric space $(\mathcal{W}_p(\bbR^d),\overline{SW}_p)$. Theorem~\ref{thm:BG21} implies
  that the identity mapping is continuous from $(\mathcal{W}_p(\bbR^d),\overline{SW}_p)$
   into $(\mathcal{W}_p(\bbR^d),\mathcal{W}_p)$. The continuous mapping theorem  implies that $\hat F_{n,X}\to F_X$ in probability
   in the metric space $(\mathcal{W}_p(\bbR^d),\mathcal{W}_p)$. Equivalently, $\mathcal{W}_p(\hat F_{n,X},F_X)\to 0$ in probability.

\medskip
Step 3). By the triangle inequality,
\[
\mathcal{W}_p(\hat F_{n,X},F_X)\leq \mathcal{W}_p(\hat F_{n,X},\delta_0)+\mathcal{W}_p(\delta_0,F_X) 
\]
with $\delta_0$ the Dirac mass at $0$. Furthermore, for any $\mu\in \mathcal{W}_p(\bbR^d)$,
\[
\mathcal{W}_p(\mu,\delta_0)=\left(\int_{\bbR^d}\|x\|^p\,\mu(\rmd x)\right)^{1/p}=M_p(\mu).
\]
We deduce
\[
\mathcal{W}_p^p(\hat F_{n,X},F_X)\leq 2^{p-1}M_p^p(\hat F_{n,X})+2^{p-1} M_p^p(F_X).
\]
In order to prove the uniform integrability of the left hand side, it is enough to prove that 
\begin{equation}\label{eq:ui}
\mbox{$M_p^p(F_X)$ is integrable and $M_p^p(\hat F_{n,X})$, $n\geq 1$, is uniformly integrable}.
\end{equation}
We have
\[
M_p^p(F_X)=\bbE[\|Y\|^p\mid X]
\]
which is integrable because $\bbE[\|Y\|^p]<\infty$. Furthermore, 
\[
M_p^p(\hat F_{n,X})=\sum_{i=1}^n W_{ni}(X)\|Y_i\|^p
\]
and Stone's Theorem ensures that 
\[
\sum_{i=1}^n W_{ni}(X)\|Y_i\|^p\longrightarrow \bbE[\|Y\|^p\mid X] \quad \mbox{in $L^1$}. 
\]
Since the sequence $M_p^p(\hat F_{n,X})$ converges in $L^1$, it is uniformly integrable and the claim follows.
\end{proof}

\subsection{Proof of Proposition~\ref{prop:upperbound}, Corollaries~\ref{cor:kernel}-\ref{cor:knn} and Theorem~\ref{thm:minimax}}

    \begin{proof}[Proof of Proposition~\ref{prop:upperbound}]
       The proof of the upper bound relies on  a coupling argument. Without loss of generality, we can assume that the $Y_i$'s are generated from  uniform random variables $U_i$'s  by the inversion method -- i.e.  we assume that $U_i$, $1\leq i\leq n$, are independent identically distributed random variables with uniform distribution on $(0,1)$ that are furthermore independent from the covariates $X_i$, $1\leq i\leq n$ and we set $Y_i=F^{ -1}_{X_i}(U_i)$. Then the sample $(X_i,Y_i)$ is i.i.d. with distribution $P$. In order to compare $\hat F_{n,x}$ and $F_x$, we introduce the random variables $\tilde Y_i=F^{-1}_{x}(U_i)$ and we define
    \[
    \tilde F_{n,x}(z)=\sum_{i=1}^n W_{ni}(x)\mathds{1}_{\{\tilde Y_i\leq z\}}.
    \]
    By the triangle inequality,
  	\[
	\mathcal{W}_1(\hat F_{n,x}, F_x)\leq 	\mathcal{W}_1(\hat F_{n,x}, \tilde F_{n,x})+	\mathcal{W}_1(\tilde F_{n,x}, F_x).
	\]
    In the right hand side, the first term is interpreted as an \textit{approximation error} comparing the weighted sample $(Y_i,W_{ni}(x))$ to $(\tilde Y_i,W_{ni}(x))$ where the $\tilde Y_i$ have the target distribution $F_x$. The second term is an \textit{estimation error} where we use the weighted sample $(\tilde Y_i,W_{ni}(x))$ with the correct distribution to estimate $F_x$.
    
    We first consider the approximation error. A similar argument as for the proof of Equation~\eqref{eq:coupling-bound} implies
    \[
    \mathcal{W}_1(\hat F_{n,x}, \tilde F_{n,x})\leq \sum_{i=1}^n W_{ni}(x)|Y_i-\tilde Y_i|.
    \]
    Introducing the uniform random variables $U_i$'s, we get
    \begin{align*}
	\bbE[ \mathcal{W}_1(\hat F_{n,x}, \tilde F_{n,x})]& \leq \bbE\Big[ \sum_{i=1}^n W_{ni}(x)|F^{-1}_{X_i}(U_i)-F^{-1}_{x}(U_i) |\Big] \\
 &=\bbE\Big[ \sum_{i=1}^n W_{ni}(x)\displaystyle\int_0^1|F^{-1}_{X_i}(u)-F^{-1}_{x}(u)|~\rmd u \Big] \quad \text{by independence}\\
	&= \bbE \Big[  \sum_{i=1}^n W_{ni}(x)\mathcal{W}_1(F_{X_i},F_{x})\Big],
	\end{align*}
    where the equality relies on Equation~\eqref{eq:wasserstein1}. Note that this control of the approximation error is very general and could be extended to the Wasserstein distance of order $p>1$.
    
    We next consider the estimation error and our approach works for $p=1$ only. By Equation~\eqref{eq:wasserstein2},
    \[
	\bbE[\mathcal{W}_1(\tilde F_{n,x}, F_x)] =\bbE\Big[   \int_{\bbR} \Big|\sum_{i=1}^n W_{ni}(x)\big(\mathds{1}_{\{\tilde Y_i\leq z\}}-F_x(z) \big)\Big|\mathrm{d}z\Big].
	\]
	Applying Fubini's theorem and using the upper bound
    \begin{align*}
       & \bbE\Big[ \Big|\sum_{i=1}^n W_{ni}(x)\big(\mathds{1}_{\{\tilde Y_i\leq z\}}-F_x(z) \big)\Big|\Big]\\
       \leq&\;  \bbE\Big[ \Big|\sum_{i=1}^n W_{ni}(x)\big(\mathds{1}_{\{\tilde Y_i\leq z\}}-F_x(z) \big)\Big|^2\Big]^{1/2}\\
       =&\; \bbE\Big[\sum_{i=1}^n W_{ni}^2(x)\Big]^{1/2} \sqrt{F_x(z)(1-F_x(z))},
    \end{align*}
    we deduce
    \[
    \bbE[\mathcal{W}_1(\tilde F_{n,x}, F_x)]\leq \bbE\Big[\sum_{i=1}^n W_{ni}^2(x)\Big]^{1/2} \int_{\bbR}\sqrt{F_x(z)(1-F_x(z))}\mathrm{d}z.
	\]
    Collecting the two terms yields Proposition~\ref{prop:upperbound}.
    \end{proof}
    
    \begin{proof}[Proof of Corollary~\ref{cor:kernel}]
    For the kernel algorithm with uniform kernel and weights~\eqref{eq:kernel-weights}, we denote by 
    \[
    N_n(X)=\sum_{i=1}^n \mathds{1}_{\{X_i\in B(X,h_n)\}}
    \]
    the number of points in the ball $B(X,h_n)$ with center $X$ and radius $h_n$.   If $N_n\geq 1$, only the points in $B(X,h_n)$ have a nonzero weight which is equal to $1/N_n$. If $N_n=0$, then by convention all the weights are equal to $1/n$. Thus we deduce
    \[
    \bbE\Big[\sum_{i=1}^nW_{ni}^2(X)\Big]=\bbE\Big[\frac{1}{N_n(X)}\mathds{1}_{\{N_n(X)\geq 1\}}\Big]+\frac{1}{n}\mathbb{P}(N_n(X)=0)
    \]
    and
    \[
    \bbE\Big[\sum_{i=1}^{n}W_{ni}(X)\|X_i-X\|^H\Big]\leq h_n^H \mathbb{P}(N_n(X)\geq 1)+k^{H/2}\mathbb{P}(N_n(X)=0)
    \]
    because the distance to $X$ for the points with non zero weight can be bounded from above by $h_n$ if $N_n(X)\geq 1$ and by $\sqrt{k}$ otherwise (note that $\sqrt{k}$ is the diameter of $[0,1]^k$).
    
    Next, we use the fact that, conditionally on $X=x$, $N_n(x)$ has a binomial distribution with parameters $n$ and $p_n(x)=\mathbb{P}(X_1\in B(x,h_n))$. This implies
    \[
    \bbE\Big[\frac{1}{N_n(X)}\mathds{1}_{\{N_n(X)\geq 1\}}\Big]\leq \bbE\Big[\frac{2}{np_n(X)}\Big]\leq \frac{2 c_k}{nh_n^k}
    \]
    where the first inequality follows from  \citet[Lemma~4.1]{gyorfi} and the second one  from  \citet[Equation~5.1]{gyorfi} where the constant $c_k=k^{k/2}$ can be taken.   Similarly, 
    \begin{align*}
    \mathbb{P}(N_n(X)=0)&=\mathbb{E}[(1-p_n(X))^n]\leq \mathbb{E}[e^{-np_n(X)}] \\
    &\leq \big(\max_{u>0} ue^{-u}\big)\times \bbE\Big[\frac{1}{np_n(X)}\Big]\\
    &\leq  \frac{c_k}{nh_n^k}.
    \end{align*}
    In view of these different estimates, Equation~\eqref{eq:general-bound} entails
	\begin{align*}
	 \bbE\big[\mathcal{W}_1(\hat F_{n,X}, F_X)\big]
	 & \leq L \Big(h_n^H+k^{H/2}\frac{c_k}{nh_n^k}\Big)  +M \left(\frac{(2+1/n)c_k}{nh_n^k}\right)^{1/2}\\
	 &\leq Lh_n^H + M\sqrt{(2+1/n)c_k}(nh_n^k)^{-1/2}+Lk^{H/2}c_k(nh_n^k)^{-1}. \end{align*}
    \end{proof}
    \begin{proof}[Proof of Corollary~\ref{cor:knn}]
    For the nearest neighbor weights~\eqref{eq:nearest-neighbor-weights}, there are exactly $\kappa_n$ non-vanishing weights with value $1/\kappa_n$ whence
    \[
    \sum_{i=1}^nW_{ni}^2(X)=\frac{1}{\kappa_n}.
    \]
    Furthermore, the $\kappa_n$ nearest neighbors of $X$ satisfy 
    \[
    \|X_{i:n}(X)-X\|\leq \|X_{\kappa_n:n}(X)-X\|,\quad i=1,\ldots,\kappa_n.
    \]
    In view of this, Equation~\eqref{eq:general-bound} entails
	\begin{align*}
	 \bbE\big[\mathcal{W}_1(\hat F_{n,X}, F_X)\big]
	 & \leq L  \bbE\big[\|X_{\kappa_n:n}(X)-X\|^H\big] +M \kappa_n^{-1/2}\\
	 & \leq L  \bbE\big[\|X_{\kappa_n:n}(X)-X\|^2\big]^{H/2} +M \kappa_n^{-1/2}
	\end{align*}
	where the last line relies on Jensen's inequality. We conclude thanks to  \citet[Theorem~2.4]{biau_2015} stating that
	\[
	\bbE\left[\| X_{\kappa_n:n}(X)-X \|^2\right] 
	\leq \begin{cases}
		 	8 (\kappa_n/n) & \mbox{if } k=1,\\
		 	\tilde{c}_k (\kappa_n/n)^{2/k} & \mbox{if } k\geq 2.
		 \end{cases}
    \]	
    \end{proof}

    \begin{proof}[Proof of Theorem~\ref{thm:minimax} (lower bound)] The proof of a lower bound for the minimax risk in Wasserstein distance is adapted from the proof of Proposition~3 in \citet[Appendix~C]{PDNT_22}  and we give only the main lines.
    
    Consider the subclass of $\mathcal{D}(H,L,M)$ where $Y$ is a binary variable with possible values $0$ and $B$. Note that condition c) of Definition~\ref{def:class-D} is automatically satisfied if $B\leq 4M$. The conditional distribution of $Y$ given $X=x$ is characterized by 
	\[
	p(x)=\mathbb{P}(Y=B\mid X=x)
	\]
	and the Wasserstein distance by
	\[
	\mathcal{W}_1(F_x,F_{x'})=B|p(x)-p(x')|,
	\]
	so that property b) of Definition~\ref{def:class-D} is equivalent to
     \begin{equation}\label{eq:regularity}
    B |p(x)-p(x')|\leq L \Vert x-x' \Vert^H.
     \end{equation}
	Similarly as in \citet[Lemma~1]{PDNT_22}, one can show that a general prediction  with values in $\mathbb{R}$ can always be improved (in terms of Wasserstein error) into a binary prediction with values in $\{0,B\}$. Indeed, for a given prediction $\hat F_{n,x}$, the binary prediction  
	\[
	\tilde F_{n,x} =(1-\tilde p_n(x))\delta_0 + \tilde p_n(x) \delta_B
	\]
	with 
	\[
	\tilde p_n(x)=\frac{1}{B} \int_0^B \big(1-\hat F_{n,x}(z)\big)\mathrm{d}z
	\]
    always satisfies
	\[
	\bbE[\mathcal{W}_1(\tilde F_{n,X},F_{X})]\leq \bbE[\mathcal{W}_1(\hat F_{n,X},F_{X})].
	\]
    This simple remark implies that, when considering the minimax risk on the restriction of the class $\mathcal{D}(H,L,M)$ to binary distributions, we can focus on binary  predictions. But for binary predictions, 
    \[
    \bbE[\mathcal{W}_1(\tilde F_{n,X},F_{X})]=B|\tilde p_{n}(X)-p(X)|,
    \]
    showing that the minimax rate of convergence for distributional regression in Wasserstein distance is equal to the minimax rate of convergence for estimating the regression function $\mathbb{E}[Y|X=x]=B p(x)$ in absolute error under the regularity assumption~\eqref{eq:regularity} . According to \cite{Stone1980, Stone1982}, a lower bound for the minimax risk in $L^1$-norm is  $n^{-H/(2H+k)}$ (in the first paper, we consider the Bernoulli regression model referred to as Model 1 Example 5 and the $L^q$ distance with $q=1$).
    \end{proof}

    \begin{proof}[Proof of Theorem~\ref{thm:minimax} (upper bound)]
    For the kernel method, Corollary~\ref{cor:kernel} states that the expected Wasserstein error is upper bounded by
    \[
     Lh_n^H +M\sqrt{(2+1/n)c_k}(nh_n^k)^{-1/2}+Lk^{H/2}c_k(nh_n^k)^{-1}.
    \]
    Minimizing the sum of the first two terms in the right-hand side with respect to $h_n$ leads to $h_n\propto n^{1/(2H+1)}$ and implies that right-hand side is of order $n^{-H/(2H+k)}$ (the last term is negligible). This matches the minimax lower rate of convergence previously  stated previously and proves that the  optimal minimax risk is of order $n^{-H/(2H+k)}$. 
    
    For the nearest neighbor method, minimizing the upper bound for the expected Wasserstein error  from Corollary~\ref{cor:knn} leads to
    $$\kappa_n\propto \begin{cases}
        n^{H/(H+1)}& \mbox{if } k=1\\
        n^{H/(H+k/2)}&\mbox{if } k\geq2
    \end{cases},$$ 
    with a corresponding risk of order  
    $$ \begin{cases}
        n^{-H/(2H+2)}& \mbox{if } k=1\\
        n^{-H/(2H+k)}&\mbox{if } k\geq2
    \end{cases},
    $$ 
    whence the nearest neighbor method reaches the optimal rate when $k\geq2$.
    \end{proof}

\subsection{Proof of Proposition~\ref{prop:appli}}
\begin{proof}[Proof of Proposition~\ref{prop:appli}] 
The first point follows from the fact that composition by a continuous application respects convergence in probability. Indeed, as the estimator $\hat F_{n,X}$ converges to $F_X$ in probability for the Wasserstein distance $\mathcal{W}_p$, $S(\hat F_{n,X})$ converges to $S(F_X)$ in probability.


In order to prove the consistency in $\mathrm{L}^{p/q}$, it is enough to prove furthermore the uniform integrability of $|S(\hat F_{n,X})-S(F_{X})|^{p/q}$, $n\geq 1$. With the convexity inequality of power functions as $p/q\geq 1$, Equation \eqref{eq:linear-bound} entails 
\begin{align*}
|S(\hat F_{n,X})-S(F_{X})|^{p/q}&\leq 2^{p/q-1}\big(|S(\hat F_{n,X})|^{p/q}+|S(F_{X})|^{p/q}\big)\\
&\leq 2^{p/q-1}\Big((aM_p^q(\hat F_{n,X})+b)^{p/q}+(aM_p^q(F_{X})+b)^{p/q}\Big)\\
&\leq 2^{2(p/q-1)}\Big(a^{p/q} M_p^p(\hat F_{n,X})+a^{p/q} M_p^p(F_{X})+2b^{p/q}\Big).
\end{align*}
This upper bound together with Equation~\eqref{eq:ui} implies the uniform integrability of $|S(\hat F_{n,X})-S(F_{X})|^{p/q}$, $n\geq 1$, which concludes the proof.
\end{proof}

\section*{Acknowledgements}
The authors acknowledge the support of the French Agence Nationale de la Recherche (ANR) under reference ANR-20-CE40-0025-01 (T-REX project). They are also grateful to Mehdi Dagdoug for suggesting the example of random forest weights (Example~\ref{example2}).

\bibliography{references.bib}

\begin{thebibliography}{23}
\providecommand{\natexlab}[1]{#1}
\providecommand{\url}[1]{\texttt{#1}}
\expandafter\ifx\csname urlstyle\endcsname\relax
  \providecommand{\doi}[1]{doi: #1}\else
  \providecommand{\doi}{doi: \begingroup \urlstyle{rm}\Url}\fi

\bibitem[Bayraktar and Guo(2021)]{BG21}
Erhan Bayraktar and Gaoyue Guo.
\newblock {Strong equivalence between metrics of Wasserstein type}.
\newblock \emph{Electronic Communications in Probability}, 26\penalty0
  (none):\penalty0 1 -- 13, 2021.
\newblock \doi{10.1214/21-ECP383}.
\newblock URL \url{https://doi.org/10.1214/21-ECP383}.

\bibitem[Biau and Devroye(2015)]{biau_2015}
Gérard Biau and Luc Devroye.
\newblock \emph{Lectures on the Nearest Neighbor Method}.
\newblock Springer Series in the Data Sciences. Springer, 2015.

\bibitem[Bobkov and Ledoux(2019)]{BL19}
Sergey Bobkov and Michel Ledoux.
\newblock One-dimensional empirical measures, order statistics, and
  {K}antorovich transport distances.
\newblock \emph{Mem. Amer. Math. Soc.}, 261\penalty0 (1259):\penalty0 v+126,
  2019.
\newblock ISSN 0065-9266.
\newblock \doi{10.1090/memo/1259}.
\newblock URL \url{https://doi.org/10.1090/memo/1259}.

\bibitem[Breiman(2001)]{Breiman_2001}
Leo Breiman.
\newblock Random forests.
\newblock \emph{Machine Learning}, 45, 2001.
\newblock \doi{10.1023/a:1010933404324}.

\bibitem[Gneiting and Katzfuss(2014)]{gneitingkatz}
Tilmann Gneiting and Matthiass Katzfuss.
\newblock Probabilistic forecasting.
\newblock \emph{Annual Review of Statistics and its Applications}, 2014.
\newblock \doi{10.1146/annurev-statistics-062713-085831}.

\bibitem[Gneiting et~al.(2005)Gneiting, Raftery, Westveld, and
  Goldman]{gneiting_raftery_2005}
Tilmann Gneiting, Adrian~E. Raftery, Anton~H. Westveld, and Tom Goldman.
\newblock {C}alibrated {P}robabilistic {F}orecasting {U}sing {E}nsemble {M}odel
  {O}utput {S}tatistics and {M}inimum {C}{R}{P}{S} {E}stimation.
\newblock \emph{Monthly Weather Review}, 133\penalty0 (5):\penalty0 1098 --
  1118, 2005.
\newblock \doi{10.1175/MWR2904.1}.

\bibitem[Greenwood et~al.(1979)Greenwood, Landwehr, Matalas, and
  Wallis]{prob_weig_mome}
J.~Arthur Greenwood, J.~Maciunas Landwehr, N.~C. Matalas, and J.~R. Wallis.
\newblock Probability weighted moments: Definition and relation to parameters
  of several distributions expressable in inverse form.
\newblock \emph{Water Resources Research}, 15\penalty0 (5):\penalty0
  1049--1054, 1979.
\newblock \doi{https://doi.org/10.1029/WR015i005p01049}.
\newblock URL
  \url{https://agupubs.onlinelibrary.wiley.com/doi/abs/10.1029/WR015i005p01049}.

\bibitem[Gy\"orfi et~al.(2002)Gy\"orfi, Kohler, Krzyzak, and Walk]{gyorfi}
L\'aszlò Gy\"orfi, Michael Kohler, Adam Krzyzak, and Harro Walk.
\newblock \emph{A Distribution-Free Theory of Nonparametric Regression}.
\newblock Springer Series in Statistics. Springer, 2002.

\bibitem[Hamill and Colucci(1997)]{hamill_1997}
Thomas~M. Hamill and Stephen~J. Colucci.
\newblock Verification of eta–rsm short-range ensemble forecasts.
\newblock \emph{Monthly Weather Review}, 125, jun 1997.
\newblock \doi{10.1175/1520-0493(1997)125<1312:VOERSR>2.0.CO;2}.

\bibitem[Henzi et~al.(2021)Henzi, Ziegel, and Gneiting]{Henzi_et_al_2021}
Alexander Henzi, Johanna~F. Ziegel, and Tilmann Gneiting.
\newblock Isotonic distributional regression.
\newblock \emph{J. R. Stat. Soc. Ser. B. Stat. Methodol.}, 83\penalty0
  (5):\penalty0 963--993, 2021.
\newblock ISSN 1369-7412.

\bibitem[Li et~al.(2021)Li, Reich, and Bondell]{li_2021}
Rui Li, Brian~J. Reich, and Howard~D. Bondell.
\newblock Deep distribution regression.
\newblock \emph{Computational Statistics \& Data Analysis}, 159:\penalty0
  107203, 2021.
\newblock ISSN 0167-9473.
\newblock \doi{https://doi.org/10.1016/j.csda.2021.107203}.

\bibitem[Nadaraya(1964)]{nadaraya_1964}
E.~A. Nadaraya.
\newblock On estimating regression.
\newblock \emph{Theory of Probability \& Its Applications}, 9\penalty0
  (1):\penalty0 141--142, 1964.
\newblock \doi{10.1137/1109020}.
\newblock URL \url{https://doi.org/10.1137/1109020}.

\bibitem[Panaretos and Zemel(2020)]{Panaretos_Zemel_2020}
Victor~M. Panaretos and Yoav Zemel.
\newblock \emph{An invitation to statistics in {W}asserstein space}.
\newblock SpringerBriefs in Probability and Mathematical Statistics. Springer,
  Cham, 2020.
\newblock ISBN 978-3-030-38437-1; 978-3-030-38438-8.
\newblock \doi{10.1007/978-3-030-38438-8}.
\newblock URL \url{https://doi.org/10.1007/978-3-030-38438-8}.

\bibitem[Pic et~al.(2022)Pic, Dombry, Naveau, and Taillardat]{PDNT_22}
Romain Pic, Clément Dombry, Philippe Naveau, and Maxime Taillardat.
\newblock Distributional regression and its evaluation with the crps: Bounds
  and convergence of the minimax risk.
\newblock \emph{International Journal of Forecasting}, 2022.
\newblock ISSN 0169-2070.
\newblock \doi{https://doi.org/10.1016/j.ijforecast.2022.11.001}.
\newblock URL
  \url{https://www.sciencedirect.com/science/article/pii/S0169207022001443}.

\bibitem[Schulz and Lerch(2021)]{schulz_2021}
Benedikt Schulz and Sebastian Lerch.
\newblock Machine learning methods for postprocessing ensemble forecasts of
  wind gusts: A systematic comparison, 2021.
\newblock \href{https://arxiv.org/abs/2106.09512}{arXiv:2106.09512}.

\bibitem[Scornet(2016)]{scornet_2016}
Erwan Scornet.
\newblock On the asymptotics of random forests.
\newblock \emph{J. Multivariate Anal.}, 146:\penalty0 72--83, 2016.
\newblock ISSN 0047-259X.
\newblock \doi{10.1016/j.jmva.2015.06.009}.
\newblock URL \url{https://doi.org/10.1016/j.jmva.2015.06.009}.

\bibitem[Stone(1977)]{Stone_1977}
Charles~J. Stone.
\newblock Consistent nonparametric regression.
\newblock \emph{Ann. Statist.}, 5\penalty0 (4):\penalty0 595--645, 1977.
\newblock ISSN 0090-5364.
\newblock URL
  \url{http://links.jstor.org/sici?sici=0090-5364(197707)5:4<595:CNR>2.0.CO;2-O&origin=MSN}.
\newblock With discussion and a reply by the author.

\bibitem[Stone(1980)]{Stone1980}
Charles~J. Stone.
\newblock Optimal rates of convergence for nonparametric estimators.
\newblock \emph{Ann. Statist.}, 8\penalty0 (6):\penalty0 1348--1360, 1980.
\newblock ISSN 0090-5364.
\newblock URL
  \url{http://links.jstor.org/sici?sici=0090-5364(198011)8:6<1348:OROCFN>2.0.CO;2-Q&origin=MSN}.

\bibitem[Stone(1982)]{Stone1982}
Charles~J. Stone.
\newblock Optimal global rates of convergence for nonparametric regression.
\newblock \emph{Ann. Statist.}, 10\penalty0 (4):\penalty0 1040--1053, 1982.
\newblock ISSN 0090-5364.
\newblock URL
  \url{http://links.jstor.org/sici?sici=0090-5364(198212)10:4<1040:OGROCF>2.0.CO;2-2&origin=MSN}.

\bibitem[Taillardat et~al.(2019)Taillardat, Fougères, Naveau, and
  Mestre]{taillardat_2019}
Maxime Taillardat, Anne-Laure Fougères, Philippe Naveau, and Olivier Mestre.
\newblock Forest-based and semiparametric methods for the postprocessing of
  rainfall ensemble forecasting.
\newblock \emph{Weather and Forecasting}, 34, jun 2019.
\newblock \doi{10.1175/WAF-D-18-0149.1}.

\bibitem[Villani(2009)]{Villani_2009}
C\'{e}dric Villani.
\newblock \emph{Optimal transport}, volume 338 of \emph{Grundlehren der
  mathematischen Wissenschaften [Fundamental Principles of Mathematical
  Sciences]}.
\newblock Springer-Verlag, Berlin, 2009.
\newblock ISBN 978-3-540-71049-3.
\newblock \doi{10.1007/978-3-540-71050-9}.
\newblock URL \url{https://doi.org/10.1007/978-3-540-71050-9}.
\newblock Old and new.

\bibitem[Watson(1964)]{watson_1964}
Geoffrey~S. Watson.
\newblock Smooth regression analysis.
\newblock \emph{Sankhyā: The Indian Journal of Statistics, Series A
  (1961-2002)}, 26\penalty0 (4):\penalty0 359--372, 1964.
\newblock ISSN 0581572X.
\newblock URL \url{http://www.jstor.org/stable/25049340}.

\bibitem[{Ćevid} et~al.(2022){Ćevid}, {Michel}, {Näf}, {Meinshausen}, and
  {Bühlmann}]{Cevid_et_al_2021}
Domagoj {Ćevid}, Loris {Michel}, Jeffrey {Näf}, Nicolai {Meinshausen}, and
  Peter {Bühlmann}.
\newblock Distributional random forests: Heterogeneity adjustment and
  multivariate distributional regression.
\newblock \emph{Journal of Machine Learning Research}, 23\penalty0
  (333):\penalty0 1--79, 2022.
\newblock URL \url{http://jmlr.org/papers/v23/21-0585.html}.

\end{thebibliography}
\end{document}